\documentclass{article}
\usepackage{mystyle}
\begin{document}
\title{Choiceless cardinals and the continuum problem}
\author{Gabriel Goldberg}
\maketitle
\begin{abstract}
    Under large cardinal hypotheses beyond the
    Kunen inconsistency --- hypotheses so strong as to contradict the Axiom of Choice ---
    we solve several variants of the generalized continuum problem
    and identify structural features of the levels \(V_\alpha\) of the cumulative hierarchy of sets
    that are eventually periodic, alternating according to the parity of the ordinal \(\alpha\). 
    For example, if there is an elementary embedding
    from the universe of sets to itself, 
    then for sufficiently large ordinals \(\alpha\), 
    the supremum of the lengths of all wellfounded relations on \(V_\alpha\)
    is a strong limit cardinal
    if and only if \(\alpha\) is odd.
\end{abstract}
\section{Introduction}
We write \(X\leq^* Y\) to abbreviate the statement that there is a surjective partial function from \(Y\) to \(X\). 
The \textit{Lindenbaum number} of a set \(X\) is the cardinal \(\aleph^*(X) = \sup \{\eta+1 : \eta \leq^* X\}\).
(The term is due to Karagila.)
For each ordinal \(\alpha\), we define the \textit{\(\alpha\)-th Lindenbaum number}:
\[\theta_\alpha = \sup_{x\in V_\alpha}\aleph^*(x)\]
If \(\alpha\) is a successor ordinal, then \(\theta_{\alpha} = \aleph^*(V_{\alpha-1})\) is the supremum of
the lengths of all wellfounded relations on \(V_{\alpha-1}\),
and if \(\gamma\) is a limit ordinal, then \(\theta_\gamma = \sup_{\xi < \gamma} \theta_\xi\).
Thus \(\theta_\omega = \omega\) and \(\theta_{\omega+1} = \omega_1\). 

The Lindenbaum numbers provide a rough measure of the size of the levels of the cumulative hierarchy.
Assuming the Axiom of Choice, \(\aleph^*(X) = |X|^+\) for all sets \(X\), and so
the generalized continuum hypothesis is equivalent to the assertion that
\(\theta_{\omega+\xi} = \aleph_{\xi}\) for all ordinals \(\xi\). 
In this paper, however, we will avoid the Axiom of Choice in order to study large cardinal hypotheses
so strong as to contradict it. Our main theorem shows that these choiceless large cardinal hypotheses 
imply a dramatic failure of the generalized continuum hypothesis at alternating levels of the cumulative hierarchy:
\begin{thm*}
    Assume there is an elementary embedding from the universe of sets to itself.
    Then for all sufficiently large limit ordinals \(\gamma\) and all natural numbers \(n\), the following hold:
    \begin{enumerate}[(1)]
        \item \(\theta_{\gamma+2n}\) 
        is a strong limit cardinal: that is, for no \(\eta < \theta_{\gamma+2n}\) is \(\theta_{\gamma+2n}\leq^* P(\eta)\).
        \item \(\theta_{\gamma+2n+1}\) is not a strong limit cardinal: in fact,
        \(\theta_{\gamma+2n+1}\leq^* P(\theta_{\gamma+2n})\).
    \end{enumerate}
\end{thm*}
This is the content of Theorems \ref{thm:even} and \ref{thm:odd}.
The main plan of attack is sketched 
in \cref{section:periodicity}
following the statements of \cref{thm:strong_suzuki}
and \cref{UltrafilterUndefinabilityThm}.

An ordinal \(\epsilon\) is \textit{even} if \(\epsilon = \gamma+2n\) for some limit ordinal \(\gamma\) and natural number \(n\);
all other ordinals are \textit{odd}. The theorem above highlights a key difference
between the structure of the even and odd levels of the cumulative hierarchy under
choiceless large cardinal hypotheses. It also indicates an analogy with the Axiom of Determinacy (AD), which has
similar ramifications for much smaller Lindenbaum numbers: under AD,
the cardinal \(\theta_{\omega+2}\), which is usually denoted by \(\Theta\), is a strong limit cardinal
whereas \(\theta_{\omega+3} = \Theta^+\). The former is a theorem of Moschovakis \cite{Moschovakis}, 
while the latter is \cref{prp:AD}, a 
simple observation of the author's.

The first instances of the periodicity phenomena around choiceless cardinals 
were discovered by Schlutzenberg and the author independently \cite{BergBerg}.
\begin{thm}[Goldberg, Schlutzenberg]\label{thm:periodicity}
    Suppose \(\alpha\) is an ordinal. If \(j : V_\alpha\to V_\alpha\) is an elementary embedding,
    \(j\) is definable over \(V_{\alpha}\) from parameters if and only if \(\alpha\) is odd.\qed
\end{thm}
A notable distinction between this theorem and the main theorem of this paper is
that the periodic properties identified here make no reference to 
metamathematical notions like elementary embeddings
and definability nor to any large cardinal theoretic concepts
whatsoever.

The calculation of the Lindenbaum numbers requires the development of a 
good deal of machinery for choiceless cardinals,
which yields some other theorems. For example, recall that the \textit{Hartogs number}
of a set \(X\), denoted by \(\aleph(X)\), is the least ordinal
that admits no injection into \(X\).
\begin{repthm}{thm:hartogs}
    If \(\epsilon\) is even and there is an elementary embedding from \(V_{\epsilon+3}\) to itself,
    then \(\aleph(V_{\epsilon+2}) = \theta_{\epsilon+2}\).
\end{repthm}
In other words, there is no \(\theta_{\epsilon+2}\)-sequence of distinct subsets of \(V_{\epsilon+1}\).  

\begin{repthm}{thm:uf_bound}
    If there is an elementary embedding from the universe of sets to itself, then
    for some cardinal \(\kappa\), for all even \(\epsilon\geq \kappa\), if \(\eta < \theta_{\epsilon+2}\), then
    the set of \(\kappa\)-complete ultrafilters on \(\eta\)
    has cardinality less than \(\theta_{\epsilon+2}\).
\end{repthm}
These theorems are analogous to consequences of the determinacy of real games (\(\text{AD}_\mathbb R\)),
which implies that \(\aleph(V_{\omega+2}) = \Theta\) and
that for all \(\eta < \Theta\),
the set of ultrafilters on \(\eta\)
has cardinality less than \(\Theta\). 
The latter is a theorem of Kechris \cite{KechrisCut}, while the former
is part of the folklore, perhaps first observed by Solovay. Both these
theorems require the hypothesis \(\text{AD}_\mathbb R\) since
their conclusions are false in \(L(\mathbb R)\). Similarly, our theorems quoted above
are false in \(L(V_{\epsilon+1})\).
\section{Background: limit Lindenbaum numbers and AD}
The following lemma shows that for limit levels of the cumulative hierarchy,
our main theorem is a simple consequence of ZF. 
\begin{lma}\label{lma:limit}
    Suppose \(\gamma\) is a limit ordinal. Then \(\theta_{\gamma}\) is a strong limit cardinal and \(\theta_{\gamma+1} = \theta_\gamma^+\).
    \begin{proof}
        To see that \(\theta_\gamma\) is a strong limit cardinal, fix \(\eta < \theta_\gamma\).
        For some \(\alpha < \gamma\), \(\eta \leq^* V_\alpha\), and so \(P(\eta) \leq V_{\alpha+1}\),
        Therefore if \(\nu \leq^* P(\eta)\), then \(\nu \leq^* V_{\alpha+1}\), and hence \(\nu < \theta_\gamma\).

        To see that \(\theta_{\gamma+1} = \theta_\gamma^+\), suppose \(f : V_{\gamma}\to \nu\)
        is a surjection, and we will show that \(|\nu| \leq \theta_\gamma\). Note that
        \[\nu = \bigcup_{\alpha < \gamma}f[V_\alpha]\] 
        and it is easy to see that 
        \(|\nu| \leq \gamma\cdot \sup_{\alpha < \gamma} \ot(f[V_\alpha])\).
        Since \(f[V_\alpha]\) is a wellorderable set that is the surjective image of \(V_\alpha\),
        \(\ot(f[V_\alpha]) < \theta_\gamma\). Hence \(|\nu| \leq \gamma\cdot \theta_\gamma = \theta_\gamma\).
    \end{proof}
\end{lma}
If \(\epsilon\) is an even successor ordinal, we do not
have nearly enough slack to generalize 
the proof above that \(\theta_\gamma\) is a strong limit cardinal
to \(\theta_\epsilon\). Moreover, 
to prove \(\theta_{\epsilon+1} = \theta_\epsilon^+\) would
seem to require a hierarchy on \(V_{\epsilon}\) similar to the rank hierarchy on
\(V_\gamma\). Under the Axiom of Determinacy, such a hierarchy
naturally presents itself:
\begin{prp}[AD]\label{prp:AD}
    Let \(\Theta = \theta_{\omega+2}\). Then \(\theta_{\omega+3} = \Theta^+\).
    \begin{proof}
        For simplicity, assume the Axiom of Dependent Choice.
        Then by the Martin-Monk theorem \cite{VanWesep}, the subsets of \(V_{\omega+1}\)
        are arranged in a wellfounded hierarchy according to their position
        in the Wadge order of continuous reducibility. 
        For \(\alpha < \Theta\), let \(\Gamma_\alpha\) be the set of subsets of \(V_{\omega+1}\)
        of Wadge rank \(\alpha\). By Wadge's lemma, \(\Gamma_\alpha \leq^* V_{\omega+1}\): 
        indeed, for any \(A\subseteq V_{\omega+1}\)
        outside \(\Gamma_\alpha\),
        \(\Gamma_\alpha\) is contained in the set of continuous preimages of \(A\).
        
        Suppose \(\nu\leq^* V_{\omega+2}\) is an ordinal,
        and we will show \(\nu < \Theta^+\).
        Fixing a surjection \(f : V_{\omega+2}\to \nu\),
        \(\nu = \bigcup_{\alpha < \Theta}f[\Gamma_\alpha]\).
        Since \(f[\Gamma_\alpha]\leq^* V_{\omega+1}\),
        \(\ot(f[\Gamma_\alpha]) < \Theta\).
        It follows that \(|\nu| \leq \Theta\), so \(\nu < \Theta^+\).

        To prove the proposition without Dependent Choice, for each \(A\subseteq V_{\omega+1}\), 
        let \(\delta(A)\) be the supremum of the lengths all wellfounded
        relations on \(V_{\omega+1}\)
        that are continuously reducible to \(A\), and let
        \(\Lambda_\alpha\) be the set of all \(A\subseteq V_{\omega+1}\)
        with \(\delta(A) \leq \alpha\).
        As above, Wadge's lemma implies that for \(\alpha < \Theta\), 
        \(\Lambda_\alpha\leq^* V_{\omega+1}\). The argument of the previous
        paragraph can be pushed through by replacing \(\Gamma_\alpha\)
        with \(\Lambda_\alpha\).
    \end{proof}
\end{prp}
\section{Periodicity for embeddings}\label{section:periodicity}
The proof in \cite{BergBerg} of \cref{thm:periodicity} 
rests on the natural attempt to define, for each ordinal \(\alpha\), 
the extension ``by continuity'' of a \(\Sigma_1\)-elementary embedding
\(j : V_\alpha \to V_\alpha\) to \(V_{\alpha+1}\).

First, let \(\mathcal H_\alpha\) be the union of all transitive sets \(M\)
such that \(M \leq^* V_\gamma\)
for some \(\gamma < \alpha\). 
The embedding \(j\) can be extended to act on \(\mathcal H_\alpha\) as follows.
Fix \(x\in \mathcal H_\alpha\), 
and let \(R\in V_\alpha\) code a wellfounded relation whose transitive collapse
is the transitive closure of \(\{x\}\). 
Then define \(j(x)\) to be the unique set \(y\) such that the transitive collapse of
\(j(R)\) is the transitive closure of \(\{y\}\).
This extends \(j\) to a well-defined \(\Sigma_0\)-elementary embedding \(j : \mathcal H_\alpha\to \mathcal H_\alpha\). 
If \(j\) is \(\Sigma_{n+1}\)-elementary 
on \(V_\alpha\),
then its extension to \(\mathcal H_\alpha\) is \(\Sigma_n\)-elementary.
Throughout this paper, when faced with an elementary embedding of \(V_\alpha\),
we will often make use of its extension to \(\mathcal H_\alpha\) without comment.

Second, we attempt to extend \(j\) to \(V_{\alpha+1}\).
If \(j : V_\alpha\to V_\alpha\) is a \(\Sigma_1\)-elementary embedding, 
the \textit{canonical extension} of \(j\)
is the map \(j^+ : V_{\alpha+1}\to V_{\alpha+1}\) defined by
\[j^+(A) = \bigcup_{N\in \mathcal H_\alpha} j(A\cap N)\]
It is not hard to show that the graph of \(j^+\) is definable over 
\(V_{\alpha+1}\) from (a code for) \(j\).
\cref{thm:periodicity} is proved by an induction that establishes:
\begin{thm}
    Suppose \(\epsilon\) is an even ordinal.
    \begin{enumerate}[(1)]
        \item If \(i : V_{\epsilon}\to V_{\epsilon}\) is elementary, 
        then \(i\) is cofinal: \(\mathcal H_{\epsilon} = \bigcup_{N\in \mathcal H_\epsilon} i(N)\).
        \item If \(j : V_{\epsilon+1}\to V_{\epsilon+1}\) is elementary,
    then \(j = (j\restriction V_{\epsilon})^+\).\qed
    \end{enumerate}
\end{thm}
We require a strong form of undefinability for elementary embeddings 
of even levels of the cumulative hierarchy (\cref{thm:strong_suzuki}), 
which can be viewed as a converse
to the following lemma:
\begin{lma}\label{lma:super_surj}
    Suppose \(j : M\to N\) is an elementary embedding
    between two transitive structures. Suppose \(X,Y\in M\)
    and \(M\) satisfies \(X\leq^* Y\). If \(j[Y]\in N\),
    then \(j[X]\) is definable over \(N\) from \(j[Y]\) and parameters
    in \(j[M]\).\qed
\end{lma}
In particular, for all \(\eta \leq^* Y\), \(j[\eta]\)
is definable over \(N\) from \(j[Y]\) and parameters in \(j[M]\).
Our strong undefinability theorem is as follows:
\begin{thm}\label{thm:strong_suzuki}
    If \(\epsilon\leq \alpha\) are ordinals, \(\epsilon\) is even, and \(j : V_\alpha\to V_\alpha\)
    is elementary, then for all \(\gamma < \epsilon\), 
    \(j[\theta_\epsilon]\) is not definable over
    \(\mathcal H_\alpha\) from \(j[V_\gamma]\)
    and parameters in \(j[\mathcal H_\alpha]\).
\end{thm}
The rough idea behind our analysis of the
even Lindenbaum numbers (\cref{thm:even}) is to show that 
given an even ordinal \(\epsilon\) and an elementary embedding
\(j : V_{\epsilon+2}\to V_{\epsilon+2}\),
for \(\eta < \theta_{\epsilon+2}\),
\(j[P(\eta)]\) \textit{is} definable over 
\(\mathcal H_{\epsilon+2}\) from \(j[V_{\epsilon+1}]\)
and parameters in \(j[\mathcal H_{\epsilon+2}]\), 
and therefore 
\(\theta_{\epsilon+2}\nleq^* P(\eta)\): otherwise
by \cref{lma:super_surj}, \(j[\theta_{\epsilon+2}]\)
would be definable over 
\(\mathcal H_{\epsilon+2}\) from \(j[V_{\epsilon+1}]\)
and parameters in \(j[\mathcal H_{\epsilon+2}]\), contrary to
\cref{thm:strong_suzuki}.

\cref{thm:strong_suzuki} follows from a general ultrafilter-theoretic fact:
\begin{thm}\label{UltrafilterUndefinabilityThm}
    Suppose \(U\) is an ultrafilter on a set \(X\)
    such that \(X\times X\leq^* X\) 
    and \(\kappa = \crit(j_U)\). 
    Then for any ordinal \(\eta\), if \(j_U[\eta]\in M_U\), then \(\eta \leq^{*} X\).
\end{thm}
The rough idea behind our analysis of the odd Lindenbaum numbers is to 
show that given an even ordinal \(\epsilon\) and an elementary embedding
\(j : V_{\epsilon+3}\to V_{\epsilon+3}\) 
\(j[\theta_{\epsilon+3}]\) is definable over
\(\mathcal H_{\epsilon+3}\) from \(j[\theta_{\epsilon+2}]\);
letting \(U\) be the ultrafilter on \(P(\theta_{\epsilon+2})\)
derived from \(j\) using \(j[\theta_{\epsilon+2}]\),
one then has that \(j_U[\theta_{\epsilon+3}]\in M_U\),
and so by \cref{UltrafilterUndefinabilityThm},
\(\theta_{\epsilon+3}\leq^* P(\theta_{\epsilon+2})\).

Granting \cref{UltrafilterUndefinabilityThm}, we now prove \cref{thm:strong_suzuki}.
\begin{proof}[Proof of \cref{thm:strong_suzuki}]
    Assume towards a contradiction that 
    for some \(\gamma < \epsilon\) and \(p\in \mathcal H_\alpha\), 
    \(j[\theta_\epsilon]\) is definable over
    \(\mathcal H_\alpha\) from \(j[V_\gamma]\) and \(j(p)\).
    By \cref{thm:periodicity}, if \(\epsilon = \gamma+1\),
    then \(j[V_\gamma]\) is definable over
    \(\mathcal H_\alpha\) from \(j[V_{\gamma-1}]\),
    and so we may assume \(\gamma+1 < \epsilon\) by replacing
    \(\gamma\) by \(\gamma-1\) if necessary.

    Let \(U\) be the ultrafilter on \(V_{\gamma+1}\) derived from 
    \(j\) using \(j[V_\gamma]\),
    and let
    \(k : \Ult(\mathcal H_\alpha,U)\to \mathcal H_\alpha\) be the factor embedding
    given by \(k([f]_U) = j(f)(j[V_\gamma])\).
    Fix a formula \(\varphi\) such that
    \[j[\theta_{\epsilon}] = \{\xi < \theta_\alpha : \mathcal H_\alpha\vDash \varphi(\xi,j[V_\gamma],j(p))\}\]
    For each \(x\in V_{\gamma+1}\), let \(f(x) = \{\xi < \theta_\alpha : \mathcal H_\alpha\vDash \varphi(\xi,x,p)\}\),
    so that 
    \[[f]_U = k^{-1}[j[\theta_{\epsilon}]] = j_U[\theta_\epsilon]\]
    But then by \cref{UltrafilterUndefinabilityThm}, \(\theta_\epsilon\leq^* V_{\gamma+1}\),
    contrary to the fact that \(\gamma+1 < \epsilon\).
\end{proof}

Though \cref{UltrafilterUndefinabilityThm} itself is purely combinatorial,  
its proof uses the techniques of ordinal definability and forcing.
Suppose \(Y\) is a set and \(X\) is ordinal definable from \(Y\).
Let \(\aleph^*_Y(X)\) denote the least ordinal
\(\delta\) such that there is no surjection
from \(X\) to \(\delta\) that is definable from \(Y\)
and ordinal parameters.
\begin{lma}[Vopenka]\label{lma:vopenka}
    Suppose \(Y\) is a set and \(X\) is ordinal definable from \(Y\).
    Then for any \(x\in X\),
    \(\HOD_{Y,x}\) is a \(\aleph^*_Y(X)\)-cc generic extension of
    \(\HOD_Y\).\qed
\end{lma}
We will only use \cref{lma:vopenka} to conclude that \(\aleph^*_Y(X)\) is regular in 
\(\HOD_{Y,x}\) for all \(x\in X\) and every stationary subset of
\(\aleph^*_Y(X)\) in \(\HOD_{Y}\) remains stationary in \(\HOD_{Y,x}\).
Both these properties are easy to verify combinatorially
assuming \(\aleph^*_Y(X)\) is regular in \(\HOD_Y\).
The \(\HOD_Y\)-regularity of \(\aleph^*_Y(X)\) seems to require some assumption on \(Y\);
we will assume that there is an \(\OD_Y\) surjection from \(X\) onto \(X\times X\).
\begin{proof}[Proof of \cref{UltrafilterUndefinabilityThm}]
        Fix a function \(f\) on \(X\) such that
        \([f]_U = j_U[\eta]\) and a surjection \(p : X\to X\times X\), and let \(Y\)
        be a set from which \(f\) and \(p\) are ordinal definable. 
        Let \(\delta = \aleph^*_{Y}(X)\).
        Then \(\delta\) is regular in \(M = \HOD_{Y}\).
        Assume towards a contradiction that \(\eta \geq \delta\). 

        For each \(x\in X\), let 
        \(M_x = \HOD_{Y,x}\), so that \(M_x\) is a \(\delta\)-cc generic extension of
        \(M\) for all \(x\in X\). Let \(N = \prod_{x\in X}M_x/U\) be the ultraproduct 
        formed using only
        those functions on \(X\) that are ordinal definable from \(Y\).
        For each such function \(g\), let \([g]\) denote the element of \(N\) it represents.
        Then this ultraproduct satisfies \L o\'s's theorem in the sense that
        \(N\vDash \varphi([g])\) if and only if \(M_x\vDash \varphi(g(x))\) for \(U\)-almost all \(x\in X\).

        Since \(M\subseteq M_x\) for all \(x\in X\), 
        we can define a function \(i\) on \(M\)
        by \(i(B) = [c_B]\) where \(c_B : X\to \{B\}\) 
        is the constant function. Then \(i\) is 
        an elementary embedding from \(M\) to 
        \[H = \{a\in N : \exists B\in M\, a\in i(B)\}\]
        
        Identifying the wellfounded part of \(N\) with its transitive collapse,
        \(\delta+1\subseteq N\) since
        \(i[\delta] = [f]\cap i(\delta)\) belongs to \(N\). 
        Clearly \(\delta\) is regular in \(M_x\) for all \(x\in X\),
        and so \(i(\delta)\) is regular in \(N\). 
        Moreover \(i(\delta) = \sup i[\delta]\)
        since there is no \(\OD_{Y}\) function from \(X\) to an unbounded subset of \(\delta\).
        Since \(i[\delta]\) belongs to \(N\) and is an unbounded subset of the \(N\)-regular cardinal \(i(\delta)\),
        we must have \(\ot(i[\delta]) = i(\delta)\). In other words, \(i(\delta) = \delta\).

        Working in \(M\), fix a partition \(\mathcal S = \langle S_\alpha : \alpha < \delta\rangle\)
        of \((S^\delta_\omega)^M\) into stationary sets.
        Let \(\mathcal T = i(\mathcal S)\), so \(\mathcal T = \langle T_\alpha : \alpha < \delta\rangle\)
        is a partition of \((S^\delta_\omega)^H\) into \(H\)-stationary sets.
        For each \(x\in X\), any set in \(P(\delta)\cap M\) that is stationary in \(M\)
        remains stationary in \(M_x\) by \cref{lma:vopenka}, and so by \L o\'s's theorem, any set in
        \(P(\delta)\cap H\) that is stationary in \(H\) remains stationary in \(N\). 
        Since \(i\) is continuous at
        ordinals of cofinality \(\omega\), the set \(i[\delta]\) is 
        an \(\omega\)-closed unbounded set in \(N\).
        It follows that \(T_\kappa \cap i[\delta]\neq \emptyset\) where \(\kappa = \crit(j_U) = \crit(i)\). 
        But if \(i(\xi)\in T_\kappa\), then \(\xi\in S_\alpha\) for some \(\alpha < \delta\),
        and hence \(\xi\in T_{i(\alpha)}\), which contradicts that \(T_\kappa\) and \(T_{i(\alpha)}\) are disjoint.
\end{proof}

\section{The coding lemma}
In this section, we establish an analog of the Moschovakis coding lemma
\cite{Moschovakis}
under choiceless large cardinal hypotheses. The proof is a 
slight twist on an argument of
Woodin \cite{SEM2} establishing a similar property of
\(L(V_{\lambda+1})\) under \(I_0\) in the context of ZFC.

For ordinals \(\lambda \leq \epsilon\), let 
\(I(\lambda,\epsilon)\) abbreviate the statement that
for all \(\alpha < \lambda\) and \(A\subseteq V_{\epsilon+1}\),
for some \(A'\subseteq V_{\epsilon+1}\), there are
elementary \(j_0,j_1 : (V_{\epsilon+1},A')\to (V_{\epsilon+1},A)\)
whose distinct critical points lie between \(\alpha\) and \(\lambda\).
Admittedly, this is a somewhat contrived large cardinal hypothesis, but
it is useful because if \(\epsilon\) is even, 
\(I(\lambda,\epsilon)\) is downwards absolute to inner models containing \(V_{\epsilon+1}\)
and follows from the existence
of an elementary embedding \(j : V_{\epsilon+2}\to V_{\epsilon+2}\) with critical point \(\kappa\)
such that \(\lambda = \sup \{\kappa,j(\kappa),j(j(\kappa)),\dots\}\).
\begin{lma}
    If \(\lambda = \sup \{\kappa,j(\kappa),j(j(\kappa)),\dots\}\)
    for some elementary \(j : V_{\epsilon+2}\to V_{\epsilon+2}\)
    with critical point \(\kappa\),
    then \(I(\lambda,\epsilon)\) holds.
    \begin{proof}
        Let \(\alpha < \lambda\) be the least ordinal
        such that \(I(\lambda,\epsilon)\) fails for \(\alpha\)
        in the sense that there is some set \(A\subseteq V_{\epsilon+1}\)
        such that there is no \(A'\subseteq V_{\epsilon+1}\) admitting elementary embeddings 
        \(j_0,j_1 : (V_{\epsilon+1},A')\to (V_{\epsilon+1},A)\)
        whose distinct critical points lie between \(\alpha\) and \(\lambda\).
        Note that \(\alpha\) is definable over \(V_{\epsilon+2}\),
        and so \(j(\alpha) = \alpha\). Since \(\lambda\) is the least fixed point of 
        \(j\) above its critical point \(\kappa\), \(\kappa > \alpha\).
        
        Fix a set \(A\subseteq V_{\epsilon+1}\) such that \((\alpha,A)\) 
        witnesses the failure of \(I(\lambda,\epsilon)\).
        Let \(A_1 = j(A)\) and \(A_2 = j(j(A))\). By the elementarity of \(j\circ j\),
        \((\alpha,A_2)\) also witnesses the failure of \(I(\lambda,\epsilon)\).
        Since \(j_0 = j\restriction V_{\epsilon+1}\) is an elementary embedding
        from \((V_{\epsilon+1},A)\) to \((V_{\epsilon+1},A_1)\), the elementarity of
        \(j\) implies that
        \(j_1 = j(j_0)\) is an elementary embedding from 
        \((V_{\epsilon+1},A_1)\) to \((V_{\epsilon+1},A_2)\).
        But note that \(j_0\) is also an elementary embedding from 
        \((V_{\epsilon+1},A_1)\) to \((V_{\epsilon+1},A_2)\) since \(j(A_1) = A_2\).
        We have \(\alpha < \crit(j_0) = \kappa < j(\kappa) = \crit(j_1) < \lambda\). 
        Therefore \(j_0\) and \(j_1\) witness \(I(\lambda,\epsilon)\) for
        \((\alpha,A_2)\), which contradicts that 
        \((\alpha,A_2)\) witnesses the failure of \(I(\lambda,\epsilon)\).
    \end{proof}
\end{lma}

Throughout this section, we fix an even ordinal \(\epsilon\). 
We will establish a (slightly technical) general result with the following
consequence:
\begin{thm}\label{thm:hod_cor}
    Fix a class \(A\)
    and let \(M = L(V_{\epsilon+1})[A]\) or \(M = \HOD_{V_{\epsilon+1},A}\).
    If \(M\) satisfies \(I(\lambda,\epsilon)\) for some \(\lambda \leq \epsilon\),
    then for all \(\eta < \theta_{\epsilon+2}^M\), \(M\) satisfies
    \(P(\eta)\cap M\leq^* V_{\epsilon+1}\).
\end{thm}

\begin{defn}
    \begin{itemize}
        \item A set \(\Gamma\subseteq V_{\epsilon+2}\) is an \textit{elementary pointclass}
        if \(\Gamma\leq^* V_{\epsilon+1}\) and for all \(\Sigma_1\)-elementary embeddings \(j : V_{\epsilon}\to V_{\epsilon}\)
        and all \(A\in \Gamma\), the image and preimage of \(A\) under 
        the canonical extension of \(j\) to \(V_{\epsilon+1}\) belongs to \(\Gamma\).
        \item Let \(\text{EP}\) denote the set of all elementary pointclasses.     
        \item If \(\lambda\) is an ordinal, then \textit{\(\lambda\)-pointclass choice} holds
        if for any any total relation \(R\) on \(\text{EP}\times V_{\epsilon+2}\),
        there is a 
        sequence \(\langle\Gamma_\alpha\rangle_{\alpha < \lambda}\in \text{EP}^{\lambda}\)
        such that \(\bigcup_{\alpha < \lambda} \Gamma_\alpha \leq^* V_{\epsilon+1}\)
        and for all \(\beta < \lambda\),
        there is some \(A\in \Gamma_\beta\) such that \(R(\bigcup_{\alpha < \beta}\Gamma_\alpha,A)\).
        \item The \textit{weak coding lemma} states that
        for any 
        surjective \(\varphi : V_{\epsilon+1}\to \eta\),
        there is an elementary pointclass \(\Gamma\) such that
        every binary relation \(R\) on \(V_{\epsilon+1}\) with \(\sup \varphi[\dom(R)] = \eta\)
        has a subrelation
        \(S\in \Gamma\) such that \(\sup \varphi[\dom(S)] = \eta\).
    \end{itemize}
\end{defn}

Any model \(M\) as in \cref{thm:hod_cor}
satisfies \(\lambda\)-pointclass choice since \(M\) contains a
well\-order\-ed set of of elementary pointclasses whose union is
\(V_{\lambda+2}\).
\begin{lma}
    Assume \(I(\lambda,\epsilon)\)
    and \(\lambda\)-pointclass choice. 
    Then the weak coding lemma holds.
    \begin{proof}
        Suppose the weak coding lemma fails. Fix a surjection \(\varphi : V_{\epsilon+1}\to \eta\) 
        witnessing this.
        
        By \(\lambda\)-pointclass choice,
        there is a sequence \(\langle \Gamma_\alpha\rangle_{\alpha < \lambda}\in \mathcal H_{\epsilon+2}\)
        such that for each \(\beta < \lambda\), there is a binary relation \(R\)
        on \(V_{\epsilon+1}\) in \(\Gamma_\beta\) with \(\sup \varphi[\dom(R)] = \eta\) that has no subrelation
        \(S\in \bigcup_{\alpha <\beta} \Gamma_\alpha\) such that \(\sup\varphi[\dom(S)] = \eta\).

        Let \(M\in \mathcal H_{\epsilon+2}\) be a transitive set such that \(V_{\epsilon+1},\eta\in M\) and
        \(\langle \Gamma_\alpha\rangle_{\alpha < \lambda}\) and \(\varphi\) belong to \(M\).
        It is left as an exercise for the reader to check that
        \(I(\lambda,\epsilon)\) implies the existence of a transitive set \(M'\) with \(V_{\epsilon+1},\eta\in M'\)
        admitting an elementary embeddings \(j_0,j_1 : M'\to M\) such that \(\crit(j_0) < \crit(j_1) < \lambda\),
        \(j_0(\eta) = j_1(\eta) = \eta\), and for some
        \(\langle \Gamma_\alpha'\rangle_{\alpha < \lambda}\) and \(\varphi'\) in \(M'\),
        \(j_i(\langle \Gamma_\alpha'\rangle_{\alpha < \lambda}) = \langle \Gamma_\alpha\rangle_{\alpha < \lambda}\)
        and \(j_i(\varphi') = \varphi\) for \(i = 0,1\).

        Let \(\kappa = \crit(j_0)\).
        By elementarity, there is a relation \(R\) in \(\Gamma_{\kappa}'\) with \(\sup \varphi'[\dom(R)] = \eta\) 
        that has no subrelation \(S\in \bigcup_{\alpha < \kappa}\Gamma_\alpha'\) such that 
        \(\sup \varphi'[\dom(S)] = \eta\). Let \(R_0 = j_0(R)\), and note that \(R_0\in \Gamma_{j_0(\kappa)}\) has no 
        subrelation in \(\Gamma_{\kappa}\) such that \(\sup \varphi[\dom(S)] = \eta\).

        Since \(j_1(\kappa) = \kappa\) and \(R\in \Gamma_\kappa'\),
        \(j_1(R)\in \Gamma_{\kappa}\). Since \(\Gamma_{\kappa}\) is an elementary pointclass,
        \(R = j_1^{-1}[j_1(R)]\in \Gamma_{\kappa}\) and hence \(j_0[R]\in \Gamma_{\kappa}\).
        But \(S = j_0[R]\) is a subrelation of \(j_0(R)\) 
        such that \[\sup \varphi[\dom(S)] =\sup j\circ \varphi'[\dom(R)] = \eta\]
        and this is a contradiction.
    \end{proof}
\end{lma}

\begin{defn}
\begin{itemize}
    \item The \textit{local collection principle} states that
    every total binary relation on \(V_{\epsilon+1}\) has
    a total subrelation \(S\) such that \(\ran(S)\leq^* V_{\epsilon+1}\).

    \item The \textit{coding lemma} states that
    for any \(\eta < \theta_{\epsilon+2}\) and any 
    surjective \(\varphi : V_{\epsilon+1}\to \eta\),
    there is an elementary pointclass \(\Gamma\) such that
    every binary relation \(R\) on \(V_{\epsilon+1}\) with \(\varphi[\dom(R)] = \eta\) has a subrelation
    \(S\in \Gamma\) such that \(\varphi[\dom(S)] = \eta\).
\end{itemize}
\end{defn}
Any model \(M\) as in \cref{thm:hod_cor} satisfies the local collection principle.
\begin{thm}\label{thm:coding}
    Assume the local collection principle
    and the weak coding lemma. 
    Then the coding lemma holds.
\end{thm}

\begin{proof}
    The proof is by induction on \(\eta\). Let \(\varphi :V_{\epsilon+1}\to \eta\) be a surjection.
    By the local collection principle and our induction hypothesis, 
    there is an elementary pointclass \(\Gamma = \{A_e\}_{e\in V_{\epsilon+1}}\) 
    that witnesses the coding lemma for all \(\gamma < \eta\).

    Let \(\Lambda\) be an elementary pointclass containing \[U = \{(e,y)\in V_{\epsilon+1}\times V_{\epsilon+1} : y\in A_e\}\] that
    witnesses the weak coding lemma for \(\eta\)
    and is closed under compositions of binary relations.
    We will show that \(\Lambda\) witnesses the coding lemma for \(\eta\).
    
    Let \(R\) be a binary relation on \(V_{\epsilon+1}\) such that
    \(\varphi[\dom(R)] = \eta\). Let \(\tilde R(x,e)\) hold if 
    \(A_e\) is a subrelation of \(R\) such that
    \(\varphi[\dom(A_e)] = \varphi(x)\). By our induction hypothesis,
    \(\varphi[\dom(R)] = \eta\), and so 
    by the weak coding lemma, \(\tilde R\)  has a subrelation \(\tilde S\in \Lambda\)
    such that \(\varphi[\dom(\tilde S)] = \eta\).
    Now \(S = U \circ \tilde R\) is a subrelation of \(R\) in \(\Lambda\)
    such that \(\varphi[\dom(S)] = \eta\).
\end{proof}

\section{The Hartogs number of \(V_{\epsilon+2}\)}
The \textit{Hartogs number} of a set \(X\), denoted by \(\aleph(X)\), 
is the least ordinal \(\eta\) such that there
is no \(\eta\)-sequence of distinct elements of \(X\).
The main theorem of this section computes the Hartogs number of the even levels of the
cumulative hierarchy:
\begin{thm}\label{thm:hartogs}
    Suppose \(\epsilon\) is an even ordinal and there is an elementary embedding from \(V_{\epsilon+3}\) to itself.
    Then \(\aleph(V_{\epsilon+2})=\theta_{\epsilon+2}\).
\end{thm}
It is easily provable in ZF that \(\aleph(V_{\epsilon+2}) \geq \aleph^*(V_{\epsilon+1}) = \theta_{\epsilon+2}\),
so the main content of \cref{thm:hartogs}
is that there is no \(\theta_{\epsilon+2}\)-sequence of distinct subsets of \(V_{\epsilon+1}\).
This does not follow from the existence
of an elementary embedding from \(V_{\epsilon+2}\) to itself.

We begin by proving the following weak version of \cref{thm:hartogs}.
\begin{prp}\label{prp:theta_limit}
    Suppose \(\epsilon\) is an even ordinal and there is an elementary \(j: V_{\epsilon+3}\to V_{\epsilon+3}\).
    Then there is no sequence \(\vec \varphi = \langle \varphi_\eta : \eta < \theta_{\epsilon+2}\rangle\)
    such that for all \(\eta < \theta_{\epsilon+2}\), \(\varphi_\eta\) is a surjection from 
    \(V_{\epsilon+1}\) onto \(\eta\).
    \begin{proof}
        Assume towards a contradiction that there is such a sequence. 
        This implies \(\theta_{\epsilon+2}\) is regular
        by the standard ZFC argument that \(\aleph^*(X)\) is regular for any set \(X\).

        Let \(\vec \psi = j(\vec \varphi)\).
        For any \(\eta \in j[\theta_{\epsilon+2}]\),
        \[j[\theta_{\epsilon+2}]\cap \eta = \psi_\eta\circ j^+[V_{\epsilon+1}]\]
        It follows that \(j[\theta_{\epsilon+2}]\) is the unique 
        \(\omega\)-closed unbounded subset \(C\subseteq \theta_{\epsilon+2}\) 
        such that for all \(\eta\in C\), \(C\cap \eta = \psi_\eta\circ j^+[V_{\epsilon+1}]\);
        here we use that any other such \(\omega\)-closed unbounded set has unbounded
        intersection with \(j[\theta_{\epsilon+2}]\), which is a consequence
        of the regularity of \(\theta_{\epsilon+2}\).

        It follows that \(j[\theta_{\epsilon+2}]\) is definable in \(\mathcal H_{\epsilon+3}\)
        from \(j[V_{\epsilon}]\) and \(\vec\psi\in \ran(j)\), which contradicts 
        \cref{thm:strong_suzuki}.
    \end{proof}
\end{prp}

\begin{cor}\label{cor:hod_bound}
    Suppose \(\epsilon\) is an even ordinal and there is an elementary \(j: V_{\epsilon+3}\to V_{\epsilon+3}\).
    Then for any class \(A\), letting \(M = \HOD_{V_{\epsilon+1},A}\), 
    \(\theta_{\epsilon+3}^{M} < \theta_{\epsilon+2}\).
    \begin{proof}
        By incorporating \(j\) into \(A\), we may assume without loss of generality
        that \(j\restriction M \in M\). We then have \(j(V_{\epsilon+2}\cap M) = V_{\epsilon+2}\cap M\).

        Using \cref{thm:hod_cor}, fix a sequence \(\langle \varphi_\eta : \eta < \theta_{\epsilon+2}^M\rangle \in M\)
        such that \(\varphi_\eta : V_{\epsilon+1}\to P(\eta)\cap M\) is a surjection.
        Note that \(j\restriction P(\eta)\cap M\)
        is uniformly definable from \(j(\varphi_\eta)\) and \(j[V_\epsilon]\) in \(\mathcal H_{\epsilon+2}\)
        for \(\eta < \theta_{\epsilon+2}^M\), and
        as a consequence \(j\restriction P_\text{bd}(\theta_{\epsilon+2}^M)\cap M\) is definable from 
        \(j(\langle \varphi_\eta : \eta < \theta_{\epsilon+2}^M\rangle)\) and \(j\restriction \theta_{\epsilon+2}^M\)   
        in \(\mathcal H_{\epsilon+2}\). Since \(\theta_{\epsilon+2}^M < \theta_{\epsilon+2}\), it follows that 
        \(j\restriction P_\text{bd}(\theta_{\epsilon+2}^M)\cap M\) is definable in \(\mathcal H_{\epsilon+2}\) from \(j[V_\epsilon]\)
        and parameters in the range of \(j\).

        Let \(E\) be the \(M\)-extender of length \(\theta_{\epsilon+2}^M\) derived from \(j\), so
        \(E\) and \(j\restriction P_\text{bd}(\theta_{\epsilon+2}^M)\cap M\) are 
        essentially the same object.
        Let \(i : \theta_{\epsilon+3}^M \to \theta_{\epsilon+3}^M\) 
        be the ultrapower associated to \(E\) (using only functions in \(M\)).
        We claim \(i = j\restriction \theta_{\epsilon+3}^M\).
        Let \(k : \Ult(\theta_{\epsilon+3}^M,E)\to \Ord\)
        be the factor embedding defined by \(k([f,a]_E) = j(f)(a)\) for \(a\in [\theta_{\epsilon+2}^M]^{<\omega}\) and 
        \(f\in M\) a function from an ordinal less than \(\theta_{\epsilon+2}^M\) into \(\theta_{\epsilon+3}^M\).
        We will show that \(k\) is the identity, or equivalently that \(k\) is surjective.
        
        Fix \(\xi < \theta_{\epsilon+3}^M\),
        and let us show \(\xi\in \ran(k)\). Let \(R\in M\) be a prewellorder of \(V_{\epsilon+2}\cap M\)
        of length greater than \(\xi\). Then \(j(R)\in M\) is a prewellorder of \(V_{\epsilon+2}\cap M\)
        of length greater than \(\xi\). Fix \(A\in V_{\epsilon+2}\cap M\) such that \(\rank_{j(R)}(A) = \xi\). 
        For each extensional \(\sigma\subseteq V_{\epsilon}\), let \(j_\sigma : V_\epsilon\to V_\epsilon\)
        denote the inverse of the transitive collapse of \(\sigma\) and let \(A_\sigma\subseteq V_{\epsilon+1}\) denote the
        preimage of \(A\) under the canonical extension of \(j_\sigma\).
        Define a partial function \(g : V_{\epsilon+1}\to \theta_{\epsilon+3}^M\)
        by \(g(\sigma) = \rank_R(A_\sigma)\). Then \(g\in M\) and \(j(g)(j[V_\epsilon]) = \xi\).
        The ordertype \(\nu\) of \(\ran(g)\) is less than \(\theta_{\epsilon+2}^M\). 
        Let \(h : \nu \to \ran(g)\) be
        the increasing enumeration. Then \(h\in M\) is
        a function from an ordinal less than \(\theta_{\epsilon+2}^M\) into \(\theta_{\epsilon+3}^M\),
        and since \(\xi\in \ran(j(g)) = \ran(j(h))\), there is some 
        \(\alpha < j(\nu)\) such that \(j(h)(\alpha) = \xi\). This shows \(\xi\in \ran(k)\),
        since \(\xi = k([h,\alpha]_E)\).

        It follows that \(i = j\restriction \theta_{\epsilon+3}^M\), and so
        \(j[\theta_{\epsilon+3}^M]\) is definable in \(\mathcal H_{\epsilon+2}\)
        from \(j[V_\epsilon]\) and parameters in the range of \(j\). By \cref{thm:strong_suzuki},
        \(\theta_{\epsilon+3}^M < \theta_{\epsilon+2}\).
    \end{proof}
\end{cor}

\begin{proof}[Proof of \cref{thm:hartogs}]
    Assume towards a contradiction that \(A = \langle A_\alpha : \alpha < \theta_{\epsilon+2}\rangle\)
    is a sequence of distinct subsets of \(V_{\epsilon+1}\).
    Let \(M = \HOD_{V_{\epsilon+1},A}\).
    Since \(A\in M\), \(\theta_{\epsilon+3}^M \geq \theta_{\epsilon+2}\), contradicting \cref{cor:hod_bound}.
\end{proof}

\section{Rank Berkeley and rank reflecting cardinals}
A cardinal \(\lambda\) is \textit{rank Berkeley} if for all ordinals \(\alpha < \lambda \leq \beta\),
there is an elementary embedding from \(V_\beta\) to itself with critical point between \(\alpha\) and \(\lambda\).
The term is due to Schlutzenberg
who noticed 
that if there is an elementary embedding from the universe of sets to itself, 
then there is a rank Berkeley cardinal. (This was realized 
independently and earlier by Woodin.)

Indeed, if \(j:V\to V\) is an elementary embedding with critical point \(\kappa\), 
then we claim \(\lambda = \sup \{\kappa,j(\kappa),j(j(\kappa)),\dots\}\)
is rank Berkeley.
Assume not, towards a contradiction, and note that \(j\) 
fixes the lexicographically least pair \((\alpha,\beta)\) 
of ordinals \(\alpha < \lambda \leq \beta\) such that
\(V_\beta\) admits no elementary embedding into itself with critical 
point between \(\alpha\) and \(\lambda\). Since \(\lambda\) is the least
fixed point of \(j\) above its critical point \(\kappa\), we have \(\alpha < \kappa\).
Therefore \(j\restriction V_\beta\) witnesses that
there is an elementary embedding from \(V_\beta\) to itself with critical 
point between \(\alpha\) and \(\lambda\), contrary to our choice of \((\alpha,\beta)\).

We prefer to work with rank Berkeley cardinals over elementary embeddings
from the universe of sets to itself, partly because the former notion 
is first-order and seems to capture all of the set-theoretic content of the latter. 
Another reason for our preference is that the least rank Berkeley cardinal
is an important threshold in the choiceless theory of large cardinals, as we now
explain.

A cardinal \(\kappa\) is \textit{almost supercompact} if for all ordinals \(\xi < \kappa\leq \beta\),
for some ordinal \(\bar \beta\) between \(\xi\) and \(\kappa\), there is an elementary embedding
from \(V_{\bar \beta}\) into \(V_\beta\) that fixes \(\xi\). In the context of ZFC, every almost supercompact
cardinal is either a supercompact cardinal or a limit of supercompact cardinals.

An apparently much weaker notion is that of a \textit{rank reflecting cardinal}, a cardinal
\(\kappa\) such that for all ordinals \(\xi < \kappa\leq \beta\) and all formulas \(\varphi(x)\) in the language of set
theory, there is an ordinal \(\bar \beta\) between \(\xi\) and \(\kappa\) such that
\(V_{\bar \beta}\vDash \varphi(\xi)\) if and only if \(V_\beta\vDash \varphi(\xi)\).
Every almost supercompact cardinal is rank reflecting, and every supercompact cardinal
is a limit of rank reflecting cardinals. In fact, the existence of a proper class
of rank reflecting cardinals is provable in ZF as an easy consequence of the L\'evy-Montague reflection theorem.

We will take advantage of the following classification of almost supercompact cardinals,
which shows that rank reflection need not be so much weaker than almost supercompactness after all:
\begin{thm}\label{thm:dichotomy}
    Let \(\lambda\) denote the least rank Berkeley cardinal.
    \begin{enumerate}[(1)]
        \item A cardinal \(\kappa \leq \lambda\) is almost supercompact if and only if it is either
        supercompact or a limit of supercompact cardinals.\label{item:super}
        \item A cardinal \(\kappa\geq \lambda\) is almost supercompact if and only if it is rank reflecting.\label{item:rank_refl}
    \end{enumerate}
\end{thm}
Justified by \cref{thm:dichotomy},
the results of the following sections are stated in terms of rank reflecting cardinals 
even though we will employ 
theorems from \cite{KunenMemorial} concerning almost supercompact cardinals.
    \begin{proof}[Proof of \cref{thm:dichotomy}]
        Since \ref{item:super} will not be needed in our applications, we will only sketch
        its proof.

        Clearly any cardinal that is either supercompact or a limit of
        supercompact cardinals is almost supercompact, so we focus on the converse.
        Suppose \(\xi\) is an ordinal and \(\eta = \eta_\xi\) is the least
        ordinal greater than \(\xi\) such that for all \(\beta \geq \eta\),
        there is some \(\bar \beta\) between \(\xi\) and \(\eta\) admitting
        an elementary embedding \(\pi : V_{\bar \beta} \to V_\beta\) such that
        \(\pi(\xi) = \xi\).
        Assuming there is no rank Berkeley cardinal less than \(\eta\),
        we will show that \(\eta\) is supercompact. Since any almost supercompact
        cardinal \(\kappa\) is the supremum of the set \(\{\eta_\xi : \xi < \kappa\}\),
        it will follow that any almost supercompact cardinal less than or equal to the
        least rank Berkeley cardinal is either a supercompact cardinal or a limit
        of supercompact cardinals.
        
        For all \(\delta < \eta\),
        for all sufficiently large ordinals \(\alpha\),
        there is no \(\bar \alpha\)
        between \(\xi\) and \(\delta\) admitting an 
        elementary embedding \(\pi : V_{\bar \alpha}\to V_\alpha\) such that
        \(\pi(\xi) = \xi\) and
        \(\delta\in \ran(\pi)\): otherwise, one can
        show that for all ordinals \(\beta\),
        there some \(\bar \beta\) between \(\xi\) and \(\delta\)
        admitting an 
        elementary embedding \(\pi : V_{\bar \alpha}\to V_\alpha\) such that \(\pi(\xi) = \xi\),
        contrary to the minimality of \(\eta\).

        Fix an ordinal \(\alpha\) large enough that 
        no cardinal less than \(\eta\) is rank Berkeley in \(V_\alpha\) and
        for all \(\delta < \eta\),
        there is no \(\bar \alpha\)
        between \(\xi\) and \(\delta\) admitting an 
        elementary embedding \(\pi : V_{\bar \alpha}\to V_{\alpha}\) fixing \(\xi\)
        with \(\delta\in \ran(\pi)\). 
        Suppose 
        \(\bar \eta < \bar \alpha < \eta\) and \(\pi :V_{\bar \alpha+1}\to V_{\alpha+1}\)
        is an elementary embedding with \(\pi(\bar \eta) = \eta\). We will show that \(\crit(\pi) = \bar \eta\).

        We first claim that \(\pi[\bar \eta]\subseteq \bar \eta\).
        Otherwise, let \(\bar \delta < \bar \eta\) be such that
        \(\pi(\bar \delta) > \bar \eta\). Let \(\delta = \pi(\bar \delta)\).
        Since \(\delta < \eta\), our choice of \(\alpha\) implies that \(\bar \alpha > \delta\).
        By elementarity, there is some \(\bar {\bar \alpha} < \bar \eta\)
        admitting an elementary \(\bar \pi : V_{\bar{\bar \alpha}}\to V_{\bar \alpha}\)
        with \(\bar \delta\in \ran(\pi)\). Now \(\bar{\bar \alpha} < \bar \eta < \delta\)
        and \(\pi\circ \bar \pi : V_{\bar{\bar \alpha}}\to V_\alpha\) is an elementary embedding
        with \(\delta\in \ran(\pi\circ \bar \pi)\), contrary to our choice of \(\alpha\).

        Assume towards a contradiction that \(\crit(\pi) < \bar\eta\).
        We will show that in \(V_\alpha\), there is a rank Berkeley cardinal less than \(\bar \eta\),
        contrary to our choice of \(\alpha\).
        Since \(\pi[\bar \eta]\subseteq \bar \eta\) while \(\pi(\bar \eta) = \eta\),
        \(\bar \eta\) must have uncountable cofinality.
        Let \(\lambda = \sup \{\eta,\pi(\eta),\pi(\pi(\eta)),\dots\}\), so 
        \(\lambda\) has countable cofinality, and hence \(\lambda < \bar \eta\).
        It is easy to see that \(\lambda\) is rank Berkeley in \(V_{\bar \eta}\):
        consider the least \(\alpha < \beta\) with \(\alpha < \lambda \leq \beta < \bar \eta\)
        such that there is no elementary embedding from \(V_\beta\) to itself with critical point
        between \(\alpha\) and \(\lambda\), and note that \(\pi(\alpha) = \alpha\) and \(\pi(\beta) = \beta\),
        and so \(\pi\restriction V_\beta\) contradicts the definition of \(\alpha\) and \(\beta\).
        By elementarity, \(\lambda\) is rank Berkeley in \(V_\eta\),
        and hence \(\lambda\) is rank Berkeley in \(V_{\bar \alpha}\).
        By elementarity, \(\lambda\) is rank Berkeley in \(V_\alpha\), and this is a contradiction.

        Turning to \ref{item:rank_refl}, the fact that
        almost supercompact cardinals are rank reflecting is immediate from the definition.
        Towards the converse, let us make some definitions and observations.

        Fix an ordinal \(\xi\). We define an increasing continuous sequence of ordinals
        \((\nu_i)_{i < \gamma_\xi}\) by letting \(\nu_0\) be \(\xi\) 
        and \(\nu_{i+1}\) the least
        ordinal \(\nu\) greater than \(\nu_i\) such that for all \(\bar \nu\) 
        between \(\xi\) and \(\nu_i\), there is no elementary embedding from
        \(V_{\bar \nu}\) into \(V_{\nu}\) that fixes \(\xi\).
        The process terminates if \(\nu_{i+1}\) does not exist,
        in which case \(\gamma_\xi = i + 1\).

        We claim that \(\gamma_\xi\) exists and is less than the least rank Berkeley cardinal
        \(\lambda\).
        To see this, assume not, and let
        \(j : V_{\nu_\lambda}\to V_{\nu_\lambda}\) be an elementary embedding
        fixing \(\xi\)
        whose critical point \(\eta\) is less than \(\lambda\). 
        Then \(j\restriction V_{\nu_{\eta+1}}\) witnesses that
        there is an elementary embedding from \(V_{\nu_{\eta+1}}\) to \(V_{\nu_{j(\eta)+1}}\)
        that fixes \(\xi\).
        By the elementarity of \(j\), for some \(\bar \nu\) between \(\xi\) and \(\nu_{\eta+1}\),
        there is an elementary embedding from \(V_{\bar \nu}\) into 
        \(V_{\nu_{\eta+1}}\) that fixes \(\xi\). This contradicts the definition
        of \(\nu_{\eta+1}\).

        Now suppose \(\kappa\geq \lambda\) is a rank reflecting cardinal above \(\xi\).
        Since \(\kappa\) is rank reflecting, there is some ordinal \(\beta < \kappa\)
        such that \(V_\beta\) correctly computes
        \(\gamma_\xi\). It follows that \(V_{\beta}\) correctly computes
        \((\nu_i)_{i < \gamma_\xi}\), and hence \(\delta = \nu_{\gamma_\xi - 1}\)
        is less than \(\kappa\). But \(\delta\) has the property that
        for all \(\alpha \geq \delta\), there is some \(\bar \alpha < \delta\)
        and an elementary embedding \(\pi : V_{\bar \alpha}\to V_\alpha\) such that
        fixing \(\xi\). Since \(\xi < \kappa\) was arbitrary, \(\kappa\) is almost supercompact.
    \end{proof}
\section{The number of ultrafilters on an ordinal}
If \(X\) is a set and \(\kappa\) is a cardinal, then
\(\beta_\kappa(X)\) denotes the set of \(\kappa\)-complete
ultrafilters on \(X\).
The main theorem of this section bounds the size of this set:
\begin{thm}\label{thm:uf_bound}
    Suppose \(\lambda\) is rank Berkeley, \(\kappa \geq \lambda\) is rank reflecting,
    and \(\epsilon \geq \kappa\) is an even ordinal. Then for all \(\eta < \theta_\epsilon\),
    \(|\beta_\kappa(\eta)| < \theta_\epsilon\).
\end{thm}
This should be contrasted with the following fact \cite[Theorem 4.14]{KunenMemorial},
which combined with \cref{thm:woc} implies that \(\beta_\kappa(\eta)\) is quite large:
\begin{thm}\label{thm:filter_extension}
    Suppose \(\lambda\) is rank Berkeley
    and \(\kappa \geq \lambda\) is rank reflecting. Then 
    every \(\kappa\)-complete filter on an ordinal extends to a \(\kappa\)-complete ultrafilter.\qed
\end{thm}
In the context of ZFC, if 
every \(\kappa\)-complete filter on an ordinal extends to a \(\kappa\)-complete ultrafilter,
then \(\kappa\) is strongly compact, and in particular
for all cardinals \(\eta\) of cofinality at least \(\kappa\),
\(|\beta_\kappa(\eta)| = 2^{2^\eta}\). \cref{thm:uf_bound} tells a very different story.

The tendency of sufficiently complete ultrafilters on ordinals to behave like ordinals
themselves is the key to the results of this section.
The following theorem \cite[Lemma 3.6]{KunenMemorial} is one example of this behavior:
\begin{thm}\label{thm:uf_wo}
    Suppose \(\lambda\) is rank Berkeley and \(\kappa \geq \lambda\) is rank reflecting.
    Then for all \(\eta\geq \kappa\), \(\beta_\kappa(\eta)\) is wellorderable.\qed
\end{thm}
In \cref{thm:even}, we will use \cref{thm:filter_extension} and \cref{thm:uf_wo} in concert
to get a handle on \(\kappa\)-complete filters on ordinals, but for our purposes 
one could do without \cref{thm:filter_extension}
by uniformly replacing \(\beta_\kappa(\eta)\) 
with a fixed wellorderable set \(S\) of ultrafilters such that every \(\kappa\)-complete
filter extends to an ultrafilter in \(S\). The existence of such a set \(S\) is
much easier to prove than \cref{thm:filter_extension}. 

\begin{lma}\label{lma:od_basis}
    Suppose \(\lambda\) is rank Berkeley, \(\kappa \geq \lambda\) is rank reflecting,
    and \(\epsilon \geq \kappa\) is an even ordinal. Then for all \(\eta < \theta_{\epsilon+2}\)
    and all surjections \(\varphi : V_{\epsilon+1}\to \eta\), 
    there is a wellordered sequence \(\mathscr F\in \HOD_{V_{\epsilon+1},\varphi}\) of 
    subsets of \(P(\eta)\)
    with the following properties:
    \begin{enumerate}[(1)]
        \item Every \(U\in \beta_\kappa(\eta)\) extends some \(B\in \mathscr F\).\label{item:basis}
        \item For each \(B\in \mathscr F\), \(\left|\{ U\in \beta_\kappa(\eta) : B\subseteq U\}\right| < \lambda\).\label{item:small}
    \end{enumerate}
    \begin{proof}
        For each \(\alpha\in \Ord\),
        let \(\mathcal E_\alpha\) denote the set of elementary embeddings 
        from \(V_{\alpha}\) to itself.
        For \(\xi < \theta_{\epsilon+3}\), let 
        \[B_\xi = \{\ran(j)\cap \eta : j\in \mathcal E_{\epsilon+3},\varphi\in j[\mathcal H_{\epsilon+2}],j(\xi)= \xi\}\]
        
        We let \(\mathscr F = \langle B_\xi\rangle_{\xi < \theta_{\epsilon+3}}\).
        Note that \(\mathscr F\) is \(\OD_\varphi\). Also, each \(B_\xi\) is included in
        \(\HOD_{V_{\epsilon+1},\varphi}\): for \(j\in \mathcal E_{\epsilon+3}\) with \(\varphi\in j[\mathcal H_{\epsilon+2}]\),
        \(\ran(j)\cap \eta = \ran(\varphi\circ j^+)\), and \(j^+\) is definable over \(V_{\epsilon+1}\) by \cref{thm:periodicity}.
        It follows that \(\mathscr F\in \HOD_{V_{\epsilon+1},\varphi}\).

        By the proof of \cite[Lemma 4.6]{KunenMemorial},
        if \(U\in \beta_\kappa(\eta)\) has Ketonen rank \(\xi\),
        then \(B_\xi\subseteq U\). This shows \ref{item:basis}.

        For any \(j: V_{\epsilon+4}\to V_{\epsilon+4}\),
        \(B_\xi\) belongs to the normal fine ultrafilter on \(P(\eta)\) derived
        from \(j\) using \(j[\eta]\). 
        By \cite[Theorem 3.12]{KunenMemorial},
        it follows that for each \(\xi < \theta_{\epsilon+3}\),
        there is a partition \(\mathcal P\) of \(\eta\)
        with \(|\mathcal P| < \lambda\) 
        such that for each \(A\in \mathcal P\),
        \(B_\xi\cup \{A\}\) extends uniquely to an element of \(\beta_\kappa(\eta)\).
        As a consequence, \(\left|\{ U\in \beta_\kappa(\eta) : B\subseteq U\}\right| < \lambda\),
        establishing \ref{item:small}.
    \end{proof}
\end{lma}

\begin{proof}[Proof of \cref{thm:uf_bound}]
    Fix a family \(\mathscr F\) as in \cref{lma:od_basis}.
    In \(M = \HOD_{V_{\epsilon+1},\varphi}\),
    \(\mathscr F\) is a wellordered family of subsets of
    \(P(\eta)\) and \(P(\eta)\leq^* V_{\epsilon+1}\) by \cref{thm:hod_cor}.
    Therefore \(|\mathscr F| < \theta_{\epsilon+3}^M\).
    By \cref{cor:hod_bound}, \(\theta_{\epsilon+3}^M < \theta_{\epsilon+2}\).
    Fix an enumeration \(\langle B_\alpha\rangle_{\alpha < |\mathscr F|}\)
    of \(\mathscr F\).
    For \(\alpha < |\mathscr F|\), let \(\mathscr U_\alpha = \{U\in \beta_\kappa(\eta) : B_\alpha\subseteq U\}\),
    so that \(|\mathscr U_\alpha| < \lambda\) for all \(\alpha < |\mathscr F|\).
    Then \(\beta_\kappa(\eta) = \bigcup_{\alpha < |\mathscr F|}\mathscr U_\alpha\),
    and so  \(|\beta_\kappa(\eta)| \leq |\mathscr F|\cdot \lambda < \theta_{\epsilon+2}\).
\end{proof}
\section{The even levels}
If \(U\) is an ultrafilter on a set \(X\)
and \(h\) and \(g\) are functions on \(X\), write \(h \sqsubseteq_U g\) if
there is a function \(e\) such that
\(h(x) = e \circ g(x)\) for \(U\)-almost all \(x\in X\). 

Suppose \(\lambda\) is rank Berkeley, \(\kappa\geq \lambda\) is 
rank reflecting, and \(\eta\geq \kappa\) is an ordinal.
Our first lemma allows us to code any \(\kappa\)-wellfounded ultrafilter
on \(\eta\) as a pair \((D,F)\) where \(D\) 
is an ultrafilter on an ordinal less than \(\kappa\)
and \(F\) is a \(\kappa\)-complete filter.

\begin{lma}\label{lma:finest_partition}
    Suppose \(\lambda\) is rank Berkeley, \(\kappa\geq \lambda\) is 
    rank reflecting,
    and \(U\) is a \(\kappa\)-wellfounded ultrafilter on an ordinal \(\eta\). 

    \begin{enumerate}[(1)]
        \item There is a \(\sqsubseteq_U\)-maximal function \(g\)
        among all functions from \(\eta\) to \(\nu\) with \(\nu < \kappa\).\label{item:finest}
        \item Letting \(D = h_*(U)\) and \(k : \Ult(P(\eta),D)\to \Ult(P(\eta),U)\) be the factor embedding,
    the \(\Ult(P(\eta),D)\)-ultrafilter \(\mathcal B\) derived from \(k\) using \(\id_U\) generates
    a \(\kappa\)-complete filter.\label{item:factor}
    \end{enumerate}
\end{lma}

We will use the wellordered collection lemma \cite[Corollary 2.22]{KunenMemorial}.
\begin{thm}\label{thm:woc}
    Suppose \(\lambda\) is rank Berkeley and \(\kappa\geq \lambda\) is rank reflecting.
    Then for any family \(\mathcal F\) of nonempty sets with \(|\mathcal F| < \kappa\),
    there is a sequence \(\langle a_x : x\in V_\kappa\rangle\) such that
    for all \(A\in \mathcal F\), there is some \(x\in V_\kappa\) with \(a_x\in A\).\qed
\end{thm}

    \begin{proof}[Proof of \cref{lma:finest_partition}]
        We may assume \(\kappa\) has cofinality \(\omega\): otherwise \(\kappa\) is a limit of 
        rank reflecting cardinals of cofinality \(\omega\), 
        and one obtains the result by applying the lemma to these
        smaller cardinals combined with a simple regressive function argument. We omit the details since 
        in our applications, we will only need the special case of the lemma in which 
        \(\kappa\) is the least rank reflecting
        cardinal greater than or equal to \(\lambda\).

        We begin with \ref{item:finest}. Using the wellordered collection lemma (\cref{thm:woc}), fix functions
        \(\langle f_x : x\in V_\kappa\rangle\) from \(\eta\) to \(\kappa\) 
        such that for each \(\alpha < \kappa\),
        there is some \(x\in V_\kappa\) such that \(\alpha = [f_x]_U\).
        Then define \(g : \eta \to \kappa^{V_\kappa}\) by
        \(g(\xi) = \langle f_x(\xi) : x\in V_\kappa\rangle\).

        For any \(\nu < \kappa\) and \(h : \eta\to \nu\), 
        there is some \(x\in V_\kappa\) such that for \(U\)-almost all \(\xi < \eta\),
        \(h(\xi) = f_x(\xi) = \text{ev}_x\circ g(\xi)\) where \(\text{ev}_x : g[\eta]\to \kappa\) is given by
        \(\text{ev}_x(s) = s(x)\). To finish, it suffices to show that there is some \(A\in U\) such that \(|g[A]| < \kappa\).

        We claim \(\aleph(V_{\kappa+1}) = \kappa^+\):
        \cite[Theorem 3.13]{KunenMemorial} implies that \(\kappa^+\) is measurable, 
        but the proof that successor cardinals cannot be measurable under 
        the Axiom of Choice shows that
        if a set \(X\) carries a nonprincipal ultrafilter that is closed under \(Y\)-indexed
        intersections, then there is no injection from \(X\) to \(P(Y)\);
        since \(\kappa\) is rank reflecting, \(\kappa = \theta_\kappa\),
        and so any \(\kappa^+\)-complete ultrafilter on \(\kappa^+\)
        is closed under \(V_\kappa\)-indexed intersections by \cite[Lemma 3.5]{KunenMemorial}.
        
        Since \(\aleph(V_{\kappa+1}) = \kappa^+\), 
        \(|g[\eta]| \leq \kappa\). Since \(\cf(\kappa) = \omega\) and \(U\) is countably complete,
        \(|g[A]| < \kappa\) for some \(A\in U\).

        We now turn to \ref{item:factor}. Fix \(g\) as in \ref{item:finest}.
        We will use the fact that if \(f : \nu\to P(\eta)\), then
        \([f]_D\in \mathcal B\) if and only if for \(U\)-almost every \(\xi < \eta\),
        \(\xi\in f\circ g(\xi)\).
        
        Suppose \(\gamma < \kappa\) and \(\langle A_\beta \rangle_{\beta < \gamma}\) belong
        to the filter generated by \(\mathcal B\).
        We will show that there is a set \(A\in \mathcal B\) such that
        \(A\subseteq \bigcap_{\xi < \gamma} A_\xi\). 
        Applying the wellordered collection lemma (\cref{thm:woc}), let \(\langle f_x: x\in V_\kappa\rangle\) be functions 
        representing sets in \(\mathcal B\) such that for all \(\beta < \gamma\), there is some \(x\in V_\kappa\)
        such that \([f_x]_D\subseteq A_\beta\).

        Let \(h : \eta\to V_{\kappa+1}\) be the function 
        \[h(\xi) = \{x\in V_\kappa : \xi \in f_x\circ g(\xi)\}\]
        For all \(x\in V_\kappa\), the fact that \([f_x]_D\in \mathcal B\) implies that
        for \(U\)-almost all \(\xi < \eta\),
        \(x\in h(\xi)\).
        Since \(\aleph(V_{\kappa+1}) = \kappa^+\), \(|h[\eta]| \leq \kappa\), and so 
        since \(g\) is \(\sqsubseteq_U\)-maximal, there is some \(e : \nu\to V_{\kappa+1}\) such that
        \(e\circ g(\xi) = h(\xi)\) for \(U\)-almost all \(\xi\).
        For all \(x\in V_\kappa\), since
        \(x\in h(\xi)\) for \(U\)-almost all \(\xi < \eta\),
        \(x\in e(\alpha)\) for \(D\)-almost all \(\alpha < \kappa\).

        Define \(f : \kappa\to P(\eta)\) by setting
        \[f(\alpha) = \bigcap_{x\in e(\alpha)}f_x(\alpha)\]
        and let \(A = [f]_D\).
        For all \(x\in V_\kappa\), for \(D\)-almost all \(\alpha < \kappa\), 
        \(f(\alpha)\subseteq f_x(\alpha)\), and hence \(A\subseteq [f_x]_D\).
        This means \(A\subseteq \bigcap_{\beta < \gamma} A_\beta\).

        Finally, we claim \(A\in \mathcal B\). To see this, note
        that by the definition of \(h\), for all \(\xi < \eta\),
        \(\xi\in \bigcap_{x\in h(\xi)} f_x\circ g(\xi)\). By our choice of \(e\),
        for \(U\)-almost all \(\xi < \eta\), 
        \[f\circ g(\xi) = \bigcap_{x\in e\circ g(\xi)}f_x\circ g(\xi) = \bigcap_{x\in h(\xi)} f_x\circ g(\xi)\]
        Hence for \(U\)-almost all \(\xi < \eta\), \(\xi \in f\circ g(\xi)\),
        which means that \(A = [f]_D\in \mathcal B\).
    \end{proof}

\begin{thm}\label{thm:even}
    Suppose \(\lambda\) is rank Berkeley, \(\kappa \geq \lambda\) is rank reflecting,
    and \(\epsilon \geq \kappa\) is an even ordinal. Then \(\theta_{\epsilon}\) is a strong limit cardinal.
    \begin{proof}
        Fix an ordinal \(\eta < \theta_{\epsilon}\). Let
        \(j : V_{\epsilon+1}\to V_{\epsilon+1}\) be an elementary embedding
        such that \(j(\eta) = \eta\).
        Let \(\nu = |\beta_\kappa(\eta)|\) and fix a wellorder \(\preceq\) of \(\beta_\kappa(\eta)\)
        in \(j[\mathcal H_{\epsilon+1}]\). 
        (Note that \(P(\eta)\) and \(\beta_\kappa(\eta)\) belong to \(\mathcal H_{\epsilon+1}\).)

        For each \(\xi < \eta\), let \(E_\xi\) be the ultrafilter on \(\eta\) derived from \(j\) using \(\xi\).
        Let \(D_\xi = g_*(E_\xi)\) where \(g : \eta\to \kappa\) is minimal modulo \(E_\xi\) among all
        \(\sqsubseteq_{E_\xi}\)-maximal functions, and 
        let \(k_\xi : \Ult(P(\eta),D_\xi)\to \Ult(P(\eta),E_\xi)\)
        be the factor embedding. 
        Let \(\delta_\xi = j(g)(\xi)\), so that \(D_\xi = E_{\delta_\xi}\).
        Note that \(D_\xi\) and \(k_\xi\) are independent of the choice of \(g\).
        Let \(\mathscr B_\xi\) be the
        \(\Ult(P(\eta),D_\xi)\)-ultrafilter derived from \(k\). By \cref{lma:finest_partition} \ref{item:factor},
        \(\mathscr B_\xi\) generates a \(\kappa\)-complete filter.
        Applying
        \cref{thm:filter_extension}, let \(U_\xi\)
        be the \(\preceq\)-least \(\kappa\)-complete ultrafilter extending \(\mathscr B_\xi\),
        and let \(\upsilon_\xi\) be the \(\preceq\)-rank of \(U_\xi\).

        Now \(E_\xi\) is uniformly 
        definable in \(\mathcal H_{\epsilon+1}\) from \(\preceq\), \(j\restriction P_{\text{bd}}(\kappa)\),
        \(\delta_\xi\), and \(\upsilon_\xi\). Specifically, \(D_\xi = j_D^{-1}[U]\)
        where \(D\) is the ultrafilter on \(\kappa\) derived from \(j\restriction P_{\text{bd}}(\kappa)\)
        using \(\delta_\xi\) and \(U\) is the \(\upsilon_\xi\)-th element of \(\beta_\kappa(\eta)\)
        in the wellorder \(\preceq\). It follows that \(\langle E_\xi : \xi < \eta\rangle\)
        is definable in \(\mathcal H_{\epsilon+1}\) from \(\preceq\), \(j\restriction P_{\text{bd}}(\kappa)\),
        and the sequence of pairs of ordinals \(s = \langle (\delta_\xi,\upsilon_\xi) : \xi < \eta\rangle\).

        We claim that \(s\) is definable in \(\mathcal H_{\epsilon+1}\) from \(j[V_\epsilon]\) and 
        parameters in \(j[\mathcal H_{\epsilon+1}]\).
        Since \(s\) is a function from \(\eta\) to
        \(\kappa\times \ot(\preceq)\) and \(\ot(\preceq) < \theta_{\epsilon}\) by \cref{thm:uf_bound},
        \(s\) can be coded by a set of ordinals \(A\) whose supremum \(\alpha\)
        is strictly less than \(\theta_{\epsilon}\). 
        Fixing \(\gamma < \epsilon\) and a surjection \(\psi : V_\gamma\to \alpha\), 
        \(j[\alpha]\) is definable as usual in \(\mathcal H_{\epsilon+1}\)
        from \(j[V_\gamma]\) and \(j(\psi)\).
        Also \(A = j^{-1}[j(A)]\) is definable from \(j[\alpha]\) and \(j(A)\).
        This means that \(A\) is definable in \(\mathcal H_{\epsilon+1}\) from \(j[V_\epsilon]\) and 
        parameters in \(j[\mathcal H_{\epsilon+1}]\), and hence so is \(s\).

        Putting everything together, \(\langle E_\xi : \xi < \eta\rangle\)
        is definable in \(\mathcal H_{\epsilon+1}\) from \(j[V_\gamma]\) and
        parameters in \(j[\mathcal H_{\epsilon+1}]\), and hence so is \(j[P(\eta)]\),
        noting that for any \(S\subseteq \eta\),
        \[j(S) = \{\xi < \eta : S\in E_\xi\}\]
        Since \(j[P(\eta)]\) is definable in \(\mathcal H_{\epsilon+1}\) from \(j[V_\epsilon]\) and
        parameters in \(j[\mathcal H_{\epsilon+1}]\), there can be no surjection
        \(f : P(\eta)\to \theta_{\epsilon}\), lest
        \[j[\theta_\epsilon] = j(f)[j[P(\eta)]]\] be definable in 
        \(\mathcal H_{\epsilon+1}\) from \(j[V_\gamma]\) and parameters in
        \(\mathcal H_{\epsilon+1}\), contrary to \cref{thm:strong_suzuki}.
    \end{proof}
\end{thm}
\section{The odd levels}
The following fact is implicit in the proof of \cref{cor:hod_bound}.
\begin{lma}
    Suppose \(\epsilon\) is an even ordinal and
    \(j :V_{\epsilon+1}\to V_{\epsilon+1}\) is an elementary embedding.
    Let \(E\) be the extender of length \(\theta_\epsilon\) derived from \(j\).
    Then \(j_E(A) = j(A)\) for any bounded subset \(A\) of \(\theta_{\epsilon+1}\).\qed
\end{lma}

\begin{cor}
    If \(\epsilon\) is an even ordinal and
    \(j :V_{\epsilon+1}\to V_{\epsilon+1}\) is an elementary embedding
    with critical point \(\kappa\),
    then the set of regular cardinals in the interval 
    \((\theta_\epsilon,\theta_{\epsilon+1})\) has cardinality less than \(\kappa\).
    \begin{proof}
        The set \(R\) of regular cardinals in the
        interval  \((\theta_\epsilon,\theta_{\epsilon+1})\) does not have cardinality exactly \(\kappa\)
        since \(|R|\) is definable over \(\mathcal H_{\epsilon+1}\)
        while \(\kappa\), being the critical point of the elementary embedding
        \(j : \mathcal H_{\epsilon+1}\to \mathcal H_{\epsilon+1}\), is not.

        Assume towards a contradiction \(|R| > \kappa\)
        and that \(\delta\) is the \(\kappa\)-th
        regular cardinal in the interval
        \((\theta_\epsilon,\theta_{\epsilon+1})\). 
        Since
        \(j : \mathcal H_{\epsilon+1}\to \mathcal H_{\epsilon+1}\) is elementary,
        \(j(\delta)\) is the \(j(\kappa)\)-th regular cardinal in the same interval.

        Letting \(E\) be the extender of length \(\theta_\epsilon\) derived from \(j\),
        \(j_E(\delta) = j(\delta)\), and \(j_E\) is continuous at \(\delta\)
        since each measure of \(j_E\) lies on a cardinal smaller than \(\delta\).
        It follows that \(\cf(j(\delta)) = \cf(\sup j_E[\delta]) = \delta\),
        which contradicts that \(j(\delta)\) is a regular cardinal larger than \(\delta\).
    \end{proof}
\end{cor}

\begin{thm}\label{thm:odd}
    If \(\lambda\) is rank Berkeley, \(\kappa\geq\lambda\) is rank reflecting,
    and \(\epsilon\geq \kappa\) is an even ordinal,
    then there is a surjection from \(P(\theta_\epsilon)\) onto \(\theta_{\epsilon+1}\).
    \begin{proof}
        If \(j : V_{\epsilon+1}\to V_{\epsilon+1}\) is an elementary embedding,
        then the extender \(E\) of length \(\theta_{\epsilon}\) derived
        from \(j\) is definable over \(\mathcal H_{\epsilon+1}\)
        from \(j[\theta_\epsilon\cup j\restriction \Pbd(\kappa)]\) 
        and parameters in \(j[\mathcal H_{\epsilon+1}]\); this is
        as in \cref{thm:even},
        using the fact that \(|\bigcup_{\eta < \theta_{\epsilon}} \beta_\kappa(\eta)| =  \theta_\epsilon\)
        by \cref{thm:uf_bound}. Since \(j[\theta_{\epsilon+1}]\) can be recovered from 
        \(E\), we have that \(j_U[\theta_{\epsilon+1}]\in \Ult(V,U)\) where
        \(U\) is the ultrafilter on \(P(\Pbd(\kappa)\cup \theta_\epsilon)\) derived from 
        \(j\) using \(j[\theta_\epsilon\cup j\restriction \Pbd(\kappa)]\).
        Therefore by \cref{UltrafilterUndefinabilityThm}, there is a surjection from
        \(P(\Pbd(\kappa)\cup \theta_\epsilon)\) to \(\theta_{\epsilon+1}\).

        We may assume that \(\epsilon \geq \kappa+2\) since
        by \cref{lma:limit}, \(\theta_{\kappa+1} = \kappa^+ \leq^* P(\kappa)\).
        Since \(P(\Pbd(\kappa)\cup \theta_\epsilon) \leq V_{\kappa+1}\times P(\theta_\epsilon)\),
        there is a surjection \(f : V_{\kappa+1}\times P(\theta_\epsilon)\to \theta_{\epsilon+1}\).
        For each \(A\subseteq \theta_{\epsilon}\), let \(S_A = \{f(x,A) : x\in V_{\kappa+1}\}\), let
        \(\nu_A = \ot(S_A)\), and let
        \(g_A : \nu_A\to \theta_{\epsilon+1}\) be the increasing enumeration of \(S_A\).
        Since \(S_A\leq^* V_{\kappa+1}\), \(\nu_A < \theta_{\kappa+2}\).
        Moreover, \(\theta_{\epsilon+1} = \bigcup_{A\subseteq \theta_\epsilon}S_A\), and so the function
        \(g : P(\theta_\epsilon)\times \theta_{\kappa+2}\to \theta_{\epsilon+1}\)
        defined by \(g(A,\alpha) = g_A(\alpha)\) is a surjection. Obviously 
        \(|P(\theta_\epsilon)\times \theta_{\kappa+2}| = |P(\theta_\epsilon)|\), and so the proof is complete.
    \end{proof}
\end{thm}
\section{Questions}
Assume there is an elementary embedding from the universe of sets to itself.
\begin{qst}
    For sufficiently large even ordinals \(\epsilon\), is \(\theta_{\epsilon+1} = \theta_\epsilon^+\)?
\end{qst}
This is probably the most glaring question left open by our theorems here.
But many other combinatorial questions remain:
\begin{qst}
    For sufficiently large even ordinals \(\epsilon\), is it true that for all 
    \(\eta < \theta_{\epsilon+2}\), there is a surjection from \(V_{\epsilon+1}\) onto \(P(\eta)\)?
    Does the coding lemma hold?
\end{qst}
    Define a sequence of ordinals \(\langle \nu_\alpha : \alpha\in \Ord\rangle\)
    by setting \(\nu_{\alpha+1} = \aleph^*(P(\nu_\alpha))\) and \(\nu_\gamma = \sup_{\alpha < \gamma}\nu_\alpha\) for \(\gamma\) a limit ordinal. 
    The arguments of this paper
    show that if there is an elementary embedding from the universe of sets to itself,
    then \(\nu_\alpha\) is a strong limit cardinal for all sufficiently large ordinals \(\alpha\).
    \begin{qst}
        Is \(\theta_{\gamma+2n} = \nu_{\gamma+n}\) for all sufficiently large limit ordinals
        \(\gamma\) and all \(n < \omega\)?
    \end{qst}
    Certain arguments in this paper require simulating the Axiom of Choice
    using rank reflecting cardinals, and for this reason the following question remains open:
    \begin{qst}
        If \(\lambda\) is rank Berkeley
        --- or just assuming there is an elementary embedding from \(V_{\lambda+2}\) to itself ---
        is \(\theta_{\lambda+2}\) a strong limit cardinal?
    \end{qst}

    By analogy with the choiceless axioms, it is natural to speculate that perhaps the right higher-order 
    generalization of the Axiom of Determinacy is some kind of regularity property for subsets of \(V_{\omega+2n+1}\)
    for \(n < \omega\).
    \begin{qst}
        Is there an extension of the Axiom of Determinacy that implies \(\theta_{\omega+4}\) is a strong limit cardinal
        and \(\theta_{\omega+5}\) is its successor?
        If \(\theta_{\omega+2}\) is a strong partition cardinal, does this hold?
    \end{qst}
    \bibliography{Bibliography.bib}
    \bibliographystyle{unsrt}
\end{document}